\documentclass[11pt,reqno]{amsart}
\usepackage[T1]{fontenc}
\usepackage{lmodern}
\usepackage{amsmath,amssymb}
\usepackage{microtype}
\usepackage[hidelinks]{hyperref}

\numberwithin{equation}{section}
\newtheorem{theorem}{Theorem}[section]
\newtheorem{lemma}[theorem]{Lemma}
\newtheorem{conjecture}[theorem]{Conjecture}
\newtheorem{proposition}[theorem]{Proposition}
\newtheorem{corollary}[theorem]{Corollary}
\theoremstyle{definition}
\newtheorem{definition}[theorem]{Definition}
\theoremstyle{remark}
\newtheorem{remark}[theorem]{Remark}

\newcommand{\R}{\mathbb{R}}
\newcommand{\eps}{\varepsilon}
\DeclareMathOperator{\Id}{Id}
\DeclareMathOperator{\Scal}{Scal}

\title[Immersed fillings and constant scalar curvature]
{Immersed hypersurfaces with positive first Newton transformation}
\author{Jiangcheng You and Heng Zhang}
\date{\today}

\address{School of Mathematical Sciences, University of Science and Technology of China, Hefei, China}
\email{yjcmp@mail.ustc.edu.cn}

\address{School of Mathematical Sciences, University of Science and Technology of China, Hefei, China}
\email{hengz@mail.ustc.edu.cn}
\subjclass[2020]{Primary 53C42; Secondary 53C21, 57R42}
\keywords{Constant scalar curvature, Alexandrov immersion, Newton transformation,
partial convexity}
\hypersetup{
  pdftitle={Immersed hypersurfaces with positive first Newton transformation},
  pdfsubject={Immersed fillings and constant scalar curvature},
  pdfkeywords={constant scalar curvature, hypersurface, Alexandrov immersion,
    Newton transformation, partial convexity}}

\begin{document}
\begin{abstract}
We prove that every closed cooriented hypersurface immersion in
$\R^{n+1}$, $n\ge 2$, with positive definite first Newton transformation
bounds a compact immersed manifold with the prescribed boundary map
and outward coorientation. Combined with the rigidity results of Ros
and Pinkall, this filling theorem implies that every closed connected
hypersurface immersed in $\R^{n+1}$ with constant scalar curvature
is a round sphere.
\end{abstract}
\maketitle

\section{Introduction}\label{sec:introduction}

In the problem section of \emph{Seminar on Differential Geometry},
Yau posed the following interesting conjecture
\cite[Problem~31, p.~677]{Yau1982}.

\begin{conjecture}\label{conj:yau}
Every closed connected hypersurface immersed in $\R^{n+1}$, $n\ge2$,
with constant scalar curvature is a round sphere.
\end{conjecture}

Several important cases of Conjecture~\ref{conj:yau} have been
established. Cheng and Yau proved the convex case
\cite[Theorem~2]{ChengYau1977}. Ros established the embedded case
and the corresponding rigidity results for higher order mean
curvatures \cite{Ros1987,Ros1988}; see also
\cite{Korevaar1988,MontielRos1991}.

Further results for immersed hypersurfaces impose additional
curvature or symmetry assumptions. Ecker and Huisken
\cite{EckerHuisken1989} treated constant Weingarten curvature under
nonnegative sectional curvature; Li \cite{Li1996} obtained
pinching results, and Zheng \cite[Proposition~2]{Zheng1997}
considered nonnegative Ricci curvature. Cheng \cite{Cheng2002}
proved rigidity for compact locally conformally flat hypersurfaces
in dimensions $n>3$, while Okayasu \cite{Okayasu2005} treated
generalized rotational hypersurfaces associated with
cohomogeneity-two orthogonal actions. Related local structure
results in the conformally flat class appear in
\cite{Gururaja2026}. More recently, Li \cite{FaguiLi2026}
proved the conjecture in dimension three and obtained
higher-dimensional results under additional curvature assumptions.

The remaining case is that of general immersed hypersurfaces in
dimensions $n\ge4$, without additional curvature or symmetry
assumptions. In this paper, we resolve this remaining case and hence
confirm Conjecture~\ref{conj:yau} in all dimensions $n\ge2$.
More precisely, we prove the following theorem.

\begin{theorem}\label{thm:csc}
Let $F:M^n\to\R^{n+1}$, $n\ge2$, be an immersion of a nonempty
closed connected manifold. If the induced scalar curvature $\Scal$
is constant, then $\Scal>0$ and $F$ is a diffeomorphism onto a round
sphere of radius
$$
 r=\sqrt{\frac{n(n-1)}{\Scal}}.
$$
\end{theorem}

In particular, the induced metric is round. For a disconnected
source, the conclusion holds on each connected component, with the
same radius if the scalar curvature is constant on the whole source.

The key to Theorem~\ref{thm:csc} is the existence of a compact
immersed filling. In a remark on p.~452 of \cite{Ros1987}, Ros
credited Pinkall with the observation that the integral rigidity
argument also applies whenever the hypersurface bounds a compact
immersed manifold. Thus, to prove Theorem~\ref{thm:csc}, it suffices
to construct such a filling for every immersion under consideration,
rather than assume its existence a priori. We obtain the required
filling from a more general geometric theorem involving the first
Newton transformation.

We first fix our conventions. Throughout the paper, all manifolds
and maps are smooth.
An immersion between manifolds of equal dimension is required to
have invertible differential at every point, including boundary
points. For a cooriented immersion $F:M^n\to\R^{n+1}$ with unit
normal $\nu$, we write
$$
 A=D\nu,\qquad H=\operatorname{tr}A,\qquad P_1=H\Id-A.
$$
Thus an outward oriented sphere has positive principal curvatures,
and $H$ is the unnormalized mean curvature. The tensor $P_1$ is the
first Newton transformation \cite{Reilly1973}, and the Gauss equation
gives
$$
 \Scal=H^2-|A|^2.
$$
We identify $A$ with its associated symmetric bilinear form when
evaluating it on two tangent vectors.

In terms of the ordered principal curvatures
$\kappa_1\le\cdots\le\kappa_n$, positivity of $P_1$ is equivalent to
\begin{equation}\label{eq:partial-convexity}
 P_1>0
 \quad\Longleftrightarrow\quad
 \kappa_1+\cdots+\kappa_{n-1}>0.
\end{equation}
In the terminology of partial convexity, this is strict
$(n-1)$-convexity
\cite[Definition~3.8 and Remark~3.12]{HarveyLawson2013}.
It coincides with $2$-convexity when $n=3$ and is weaker than
$2$-convexity when $n\ge4$.

To state the filling theorem precisely, we use the following
definition.

\begin{definition}\label{def:alexandrov}
An \emph{immersed filling} of a cooriented immersion
$F:M^n\to\R^{n+1}$ consists of a compact manifold $W^{n+1}$, an
immersion $G:W\to\R^{n+1}$, and a diffeomorphism
$b:\partial W\to M$ such that $G|_{\partial W}=F\circ b$ and $dG$
maps the outward unit normal of $G^*g_{\mathrm E}$ to $\nu\circ b$.
An immersion admitting such a filling is called an
\emph{Alexandrov immersion}.
\end{definition}

The boundary identification $b$ is 
retained throughout, even when suppressed from the notation.
The filling need not be injective. Interior extension hypotheses
already occur in Alexandrov's sphere theorem for elliptic
Weingarten relations
\cite[\S\,1, conditions~$(\mathrm{II}_0)$ and~$(\mathrm{II})$]
{Alexandrov1962}. Our use of Alexandrov immersion requires
compactness; definitions allowing proper noncompact extensions
also occur \cite{LambertMaderBaumdicker2024}.

Our main filling theorem is the following.

\begin{theorem}\label{thm:filling}
Every closed cooriented immersion $F:M^n\to\R^{n+1}$, $n\ge2$,
with $P_1>0$ admits an immersed filling with the prescribed outward
coorientation.
\end{theorem}

The curvature hypothesis in Theorem~\ref{thm:filling} imposes no
additional restriction in the constant scalar curvature problem.
Indeed, Lemma~\ref{lem:admissibility} shows that every nonempty
closed connected hypersurface immersion with constant scalar
curvature has $\Scal>0$ and admits a global unit normal for which
$P_1>0$. Theorem~\ref{thm:filling} therefore provides the required
compact immersed filling, and the Ros--Pinkall rigidity argument
(Proposition~\ref{prop:ros-pinkall}) yields
Theorem~\ref{thm:csc}.

It is useful to compare Theorem~\ref{thm:filling} with earlier
filling results. Our construction is related to Gromov's method of
linear sections \cite[\S\,1/2, pp.~23--25]{Gromov1991}. The
hypothesis there counts nonnegative principal curvatures, rather
than imposing the sum condition \eqref{eq:partial-convexity}.
In \cite[\S\,5.3(c)]{Gromov2014}, Gromov also states an immersed
filling theorem under a principal-curvature sum condition.
Writing $N$ for the ambient dimension, every closed cooriented
strictly $(N-k)$-mean convex immersed hypersurface in a complete
Riemannian $N$-manifold with nonnegative sectional curvature, with
$k>N/2$, bounds a compact Riemannian manifold whose isometric
immersion into the ambient manifold extends the prescribed
boundary immersion. Here strict $(N-k)$-mean convexity means
positivity of the sum of the smallest $N-k$ principal curvatures.
In our setting, $N=n+1$, whereas
\eqref{eq:partial-convexity} involves the smallest $n-1=N-2$
principal curvatures. Thus Gromov's stated range requires a
stronger curvature hypothesis when $n\ge3$.

Huisken and Sinestrari
\cite[Corollary~1.2]{HuiskenSinestrari2009} constructed immersed
handlebody fillings for closed $2$-convex hypersurfaces of dimension
$n\ge3$. These filling results, combined with Ros's argument, also
underlie the work of F.~Li \cite{FaguiLi2026} cited above.
In dimension three, the curvature condition in
Theorem~\ref{thm:filling} agrees with $2$-convexity; in dimensions
$n\ge4$, the theorem provides fillings under the weaker condition
\eqref{eq:partial-convexity}. This is the extension of the filling
theory needed for the remaining cases of
Conjecture~\ref{conj:yau}. For related work on the topology of
manifolds with partially convex boundary, see
\cite{Sha1986,Wu1987,HarveyLawson2013}.

The proof of Theorem~\ref{thm:filling} follows a generic linear
height. Positivity of $P_1$ ensures that every regular hyperplane
section has positive mean curvature. Howard's rolling theorem
\cite{Howard1999} supplies an inward collar whose width is
controlled by the geometry of the prescribed section,
independently of the incoming filling. This gives a uniform
continuation step along each compact interval of regular heights.

At a critical height, we attach the missing side of the local graph
to complete the incoming immersion locally to a covering of a ball.
The inverse branches of this covering identify the source domain
to be replaced. Working with these source domains, rather than
with their possibly overlapping images, avoids identifying
distinct sheets. Together with the uniform collar estimates, the
local replacements allow us to assemble finitely many immersed
bands into a compact filling with the prescribed boundary
parameterization and outward coorientation.

The paper is organized as follows.
Section~\ref{sec:preliminaries} establishes the collar estimates
and the required properties of height sections.
Section~\ref{sec:critical} constructs the extension through a
critical height, and Section~\ref{sec:global} assembles the filling
and proves Theorem~\ref{thm:filling}.
Finally, Section~\ref{sec:rigidity} gives the scalar curvature
reduction and the classical integral method, completing the
proof of Theorem~\ref{thm:csc}.

\section{Collars and height sections}\label{sec:preliminaries}

\subsection{Uniform collars and boundary deformations}\label{sec:collars}

Distances and operator norms on a filling are taken in its pullback
Euclidean metric. A collar is embedded in the source; its image in
Euclidean space may overlap.

\begin{lemma}[Uniform inward collar]\label{lem:collar}
Let $G:\Omega^d\to\R^d$, $d\ge2$, be an immersion of a compact manifold
with boundary. If its outward boundary curvatures satisfy $H>0$ and
$|\kappa_i|\le K$, $K>0$, then for $\delta=(4K)^{-1}$ there is an
inward normal collar $\gamma:\partial\Omega\times[0,\delta]\to\Omega$
with distance coordinate $\sigma$ and boundary projection $\pi$ satisfying
\begin{equation}\label{eq:collar-coordinates}
 G(\gamma(p,\sigma))=G(p)-\sigma\nu(p),\qquad
 |d\sigma|=1,\qquad |D\pi|\le\frac43.
\end{equation}
\end{lemma}

\begin{proof}
Equip $\Omega$ with the flat metric $g=G^*g_{\mathrm E}$, which is
complete as a length metric by compactness. The adapted Jacobi field along an inward normal geodesic is
$$
 J_v(s)=\mathcal P_s(\Id-sA)v,
 \qquad v\in T_p\partial\Omega,
$$
where $\mathcal P_s$ denotes parallel transport. Since $|\kappa_i|\le K$,
the operator $\Id-sA$ is nonsingular for $0\le s<K^{-1}$.
Thus the boundary focal distance is at least $K^{-1}$.
Howard's rolling theorem \cite{Howard1999}, applied to each connected
component, gives the same lower bound for the inward cut distance.
A first return to the boundary at length $L$ would give a cut distance
at most $L/2$. It follows that the inward normal exponential map
defines a collar
$$
 \gamma:\partial\Omega\times[0,K^{-1})\longrightarrow\Omega.
$$

Since $G$ is a local isometry, normal geodesics map to straight lines,
so
$$
 G(\gamma(p,\sigma))=G(p)-\sigma\nu(p).
$$
On this collar, $\sigma$ is the distance to $\partial\Omega$, hence
$|d\sigma|=1$. The tangential differential of $\gamma$ is
$\mathcal P_\sigma(\Id-\sigma A)$, while $D\pi$ vanishes in the normal
direction. Therefore, for $0\le\sigma\le\delta=(4K)^{-1}$,
$$
 |D\pi|
 \le \|(\Id-\sigma A)^{-1}\|
 \le \frac{1}{1-\sigma K}
 \le \frac43.
$$
Restricting $\gamma$ to $\partial\Omega\times[0,\delta]$ proves
\eqref{eq:collar-coordinates}.
\end{proof}

A closely related cutoff extension in a normal collar appears in
\cite[Lemma~3 and equation~(4)]{LambertMaderBaumdicker2024}; its continuation
criterion is given in \cite[Corollary~4]{LambertMaderBaumdicker2024}.
Their local extension time depends on a lower bound for the initial
collar injectivity radius. Here Lemma~\ref{lem:collar} supplies such a
bound from the prescribed boundary family alone, independently of the choice of filling.

\begin{lemma}[Uniform local extension]\label{lem:local-extension}
Let $P$ be a closed smooth $(d-1)$-manifold, $d\ge2$, and let
$X_t:P\to\R^d$, $t\in I$, be a smooth family of immersions,
where $I$ is a compact interval. Suppose that the family has smooth
unit normals $\nu_t$ with $H_t>0$. There is $\tau>0$, depending only
on this family, such that for every $a\in I$, every immersed filling
$G_a:\Omega\to\R^d$ of $X_a$ extends to a jointly smooth family
of immersed fillings $G_t:\Omega\to\R^d$ of $X_t$ for
$t\in I$, $|t-a|\le\tau$, starting from the given $G_a$.
\end{lemma}

\begin{proof}
Identify $\partial\Omega$ with $P$.
For a fixed metric $h_0$ on $P$, compactness gives uniform constants
$c,K>0$ and $B_0,B_1<\infty$ such that
$$
 X_t^*g_{\mathrm E}\ge c h_0,\qquad |\kappa_i(t)|\le K,\qquad
 |\partial_tX_t|\le B_0,\qquad
 |d_P\partial_tX_t|_{h_0}\le B_1.
$$
Apply Lemma~\ref{lem:collar} to $G_a$ with $\delta=(4K)^{-1}$.
Choose a smooth cutoff $\chi:[0,\infty)\to[0,1]$, equal to one near
zero and zero near and above one. On the collar set
\begin{equation}\label{eq:boundary-extension}
 G_t=G_a+\chi(\sigma/\delta)(X_t-X_a)\circ\pi,
\end{equation}
and keep $G_a$ elsewhere. In the metric $g_a=G_a^*g_{\mathrm E}$,
\begin{equation}\label{eq:extension-estimate}
 |D(G_t-G_a)|_{g_a}
 \le\left(\frac43c^{-1/2}B_1+
 \frac{\|\chi'\|_\infty B_0}{\delta}\right)|t-a|.
\end{equation}
Choose $\tau$ to make the right-hand side less than $1/2$. Since $DG_a$
is an isometry for $g_a$, each $DG_t$ is invertible. The boundary map
is $X_t$, and its outward normal agrees with $\nu_t$ by continuity
from $t=a$.
\end{proof}

\begin{corollary}\label{cor:homotopy}
Existence of an immersed filling is preserved along a smooth regular
homotopy of closed cooriented hypersurfaces with positive mean curvature.
\end{corollary}

\begin{proof}
Subdivide the compact parameter interval into intervals of length at
most $\tau$ from Lemma~\ref{lem:local-extension}, using each endpoint
filling as the next initial filling. The boundary remains the prescribed
family, so the same time step applies at every stage.
\end{proof}

Related collar perturbations and extensions of Alexandrov immersions
appear in \cite{HauswirthKilianSchmidt2015,LambertMaderBaumdicker2024}.
The estimate \eqref{eq:extension-estimate} is the uniformity needed below.

\subsection{A generic height}\label{sec:genericity}

\begin{proposition}\label{prop:genericity}
Every closed cooriented immersion $F:M^n\to\R^{n+1}$ with $P_1>0$ admits
an arbitrarily small $C^2$ perturbation $\widetilde F$ for which some
linear height is Morse, has distinct critical values, and has a unique
preimage at each critical image. The linear homotopy from $F$ to
$\widetilde F$ consists of cooriented immersions with $P_1>0$.
\end{proposition}

\begin{proof}
Fix a background
metric on $M$. Immersivity and positivity of the symmetric form
$H_fg_f-g_f(A_f\,\cdot,\cdot)$ are open conditions in the $C^2$ topology.
A sufficiently small $C^2$ ball about $F$ therefore consists of
immersions with $P_1>0$, with unit normals chosen close to that of $F$.

First make $F$ self-transverse. To do so, choose a smooth map
$\Phi:M\to\R^L$ that separates points, for example the map formed from
cutoffs $\chi_i$ and cutoff coordinate functions $\chi_i x_i^\alpha$
in a finite atlas, where the sets on which $\chi_i=1$ cover $M$.
For $T\in\operatorname{Hom}(\R^L,\R^{n+1})$, put $F_T=F+T\Phi$.
On $M^{(2)}=(M\times M)\setminus\Delta$, the map
$$
 (T,p,q)\longmapsto F_T(p)-F_T(q)
$$
is a submersion: its parameter derivative sends $S$ to
$S(\Phi(p)-\Phi(q))$ and is onto. Sard's theorem applied to the
projection of its zero set onto the parameter space gives arbitrarily
small $T$ for which $f=F_T$ is self-transverse. This is the standard
parametric transversality argument \cite[Chapter~3, Sections~1--2]{Hirsch1976}.

The ordered double locus
$$
 D_f=\{(p,q)\in M^{(2)}:f(p)=f(q)\}
$$
is then a smooth $(n-1)$-manifold. It is compact: a finite cover by
sets on which $f$ is injective has a Lebesgue number, excluding double
pairs near the diagonal. Let $C\subset S^n$ be the critical value set
of the Gauss map $\nu_f$, and let
$E=\{\nu_f(p):(p,q)\in D_f\}$. Both have measure zero, by Sard's
theorem and the dimension of $D_f$. Choose
$$
 e\notin C\cup(-C)\cup E\cup(-E).
$$
The height $h_0=\langle f,e\rangle$ is critical precisely where
$\nu_f=\pm e$, and at such a point
$$
 \operatorname{Hess}h_0(X,Y)
 =-\langle e,\nu_f\rangle\,g_f(A_fX,Y).
$$
Since $\pm e$ are regular values of $\nu_f$, the height is Morse.
The choice of $e$ also excludes double images at its critical points.
Write these finitely many points as $p_1,\ldots,p_k$.

It remains to separate their heights. Choose pairwise disjoint graph
charts $U_i$ about $p_i$ on which the horizontal projection is injective,
and neighborhoods $W_i\Subset U_i$ of $p_i$. The unique preimage
property gives
$$
 d_i=\operatorname{dist}\bigl(f(p_i),f(M\setminus W_i)\bigr)>0.
$$
Let $\beta_i$ be supported in $U_i$ and equal to one near
$\overline W_i$, and set
$$
 f_a=f+\sum_i a_i\beta_i e,\qquad
 h_a=h_0+\sum_i a_i\beta_i.
$$
For small $a$, the critical points and their Hessians are unchanged:
on each $W_i$ the height is translated by a constant, and on the compact
complement $dh_0$ is bounded away from zero. Taking
$\|f_a-f\|_{C^0}+|a_i|<d_i/2$ excludes new preimages of $f_a(p_i)$
outside $W_i$; inside $U_i$, horizontal injectivity excludes them.
Avoiding the finitely many affine hyperplanes
$h_0(p_i)+a_i=h_0(p_j)+a_j$ makes the critical values distinct.

Put $\widetilde F=f_a$. The total perturbation can be arbitrarily small
in $C^2$, so its entire linear segment with $F$ lies in the ball fixed
at the start.
\end{proof}

At each critical point, horizontal projection gives a local graph. The
unique preimage property and compactness allow a target cylinder meeting
only that graph. We choose these cylinders pairwise disjoint, with
smaller cylinders compactly contained in them. The Morse lemma applies
to the horizontal graph function \cite[Chapter~6, Section~1]{Hirsch1976}.

\subsection{Geometry of the sections}\label{sec:height}

In the height construction below, let $F:M^n\to\R^{n+1}$, $n\ge2$,
be a closed cooriented immersion with $P_1>0$ and with the generic
height supplied by Proposition~\ref{prop:genericity}. Choose
coordinates so that
$$
 F=(X,h):M^n\longrightarrow\R^n\times\R,
 \qquad h=\langle F,e\rangle,
$$
where $e$ is the last coordinate vector. Write $g_0=F^*g_{\mathrm E}$.

At a regular point of $h$, define
$$
 v=\frac{\nabla h}{|\nabla h|},\qquad
 \alpha=|\nabla h|=|\nu-\langle\nu,e\rangle e|,
 \qquad \mu=\frac{\nu-\langle\nu,e\rangle e}{\alpha}.
$$
Then $X_t=X|_{h^{-1}(t)}$ is an immersed hypersurface in $\R^n$ with
specified normal $\mu$. If $B_t$ is its second fundamental form, then
for tangent vectors to $h^{-1}(t)$,
\begin{equation}\label{eq:slice-form}
 B_t=\alpha^{-1}A|_{T h^{-1}(t)},\qquad
 H_t=\frac{H-\langle Av,v\rangle}{\alpha}
     =\frac{\langle P_1v,v\rangle}{\alpha}>0.
\end{equation}
To verify \eqref{eq:slice-form}, differentiate $\mu$ in a direction
tangent to the section and take its inner product with another such
direction. The derivatives of $\alpha$ and of
$\langle\nu,e\rangle$ make no contribution, because those tangent
directions are orthogonal to both $\mu$ and $e$.

On a compact interval of regular values, the flow of
$\nabla h/|\nabla h|^2$ identifies the sections with a fixed compact
source. Equation~\eqref{eq:slice-form} and
Lemma~\ref{lem:local-extension} consequently apply to their prescribed
family. The induced metrics, spatial curvature, and time derivatives
have uniform bounds on any such fixed interval.

Now let a critical value be zero. Its isolated graph can be written
\begin{equation}\label{eq:critical-graph}
 F(j(z))=(z,f(z)),\qquad
 f(\Theta(x,y))=-|x|^2+|y|^2=:q(x,y),
\end{equation}
where $x\in\R^p$, $y\in\R^{n-p}$, $j$ is a source chart, and $\Theta$
is a smooth spatial coordinate change. We use the Morse lemma in the
horizontal variables.
Curvature is computed in the variable $z$; the derivatives
of $\Theta$ and $\Theta^{-1}$ are bounded on the fixed smaller chart.

Write $\lambda=1$ if $\nu$ points upwards on the graph, and
$\lambda=-1$ if it points downwards. The filling is to occupy the
closed side
\begin{equation}\label{eq:graph-side}
 \lambda(f(z)-t)\ge0.
\end{equation}
At the critical point $A=-\lambda\operatorname{Hess}f$ in orthonormal
horizontal coordinates. By \eqref{eq:partial-convexity}, $A$ has at least two positive
eigenvalues. Hence
\begin{equation}\label{eq:index-restriction}
 \lambda=1\quad\Longrightarrow\quad p\ge2.
\end{equation}
For $\lambda=-1$ one likewise has $n-p\ge2$.

We will use an incoming section at $a=-\eps$. On $q=-\eps$ one has
$|(x,y)|\ge\sqrt\eps$, and nondegeneracy gives
$|df|\ge c\sqrt\eps$ on the fixed graph chart. Since
$\alpha=|df|/(1+|df|^2)^{1/2}$ there, while $A$ is uniformly bounded,
\eqref{eq:slice-form} gives
\begin{equation}\label{eq:critical-curvature}
 |B_{-\eps}|\le C\eps^{-1/2}.
\end{equation}
Away from a fixed smaller neighborhood of the critical point,
$|\nabla h|$ is bounded below on a fixed sufficiently short height
interval, so the same estimate holds on the entire incoming section.
All constants here depend only on the prescribed immersion and the fixed
critical neighborhood. Lemma~\ref{lem:collar} therefore gives every
compact immersed filling of that section with the specified outward
normal a collar of width
\begin{equation}\label{eq:critical-collar}
 \delta=c_0\sqrt\eps,
 \qquad |d\sigma|=1,\qquad |D\pi|\le\frac43,
\end{equation}
after fixing $c_0>0$ sufficiently small. The constants are independent of the incoming filling.

\section{Extension through a critical height}\label{sec:critical}

\subsection{A local decomposition by completion}\label{sec:completion}

The following lemma identifies the part of an incoming filling adjacent
to its prescribed boundary. A closed smooth domain is understood
relative to the ambient manifold.

\begin{lemma}[Completion]\label{lem:completion}
Let a connected simply connected $n$-manifold $B$ be the union of closed
smooth domains $A_0,D_0$ with common boundary $L$ and disjoint interiors.
Assume that $D_0$ is nonempty and connected. If a proper immersion
$f:E^n\to B$ maps $\partial E$ diffeomorphically onto $L$ and has its
interior on the $A_0$-side there, then, as an immersion over $B$,
$$
 E\cong A_0\sqcup\bigsqcup_{j=1}^d B
 \qquad\text{for some }d\ge0.
$$
\end{lemma}

\begin{proof}
Attach a copy of $D_0$ to $E$ along
$\sigma=(f|_{\partial E})^{-1}:L\to\partial E$. The maps $f$ and the
inclusion of $D_0$ induce a map $\pi:Z=E\cup_\sigma D_0\to B$.
Gluing along this homeomorphism of closed boundaries gives a Hausdorff,
second countable space. At a point of the seam, a one-sided inverse
chart of $f$ on the $A_0$-side and the attached $D_0$-side together map
onto a target neighborhood. These charts make $Z$ a smooth manifold
without boundary and $\pi$ a local diffeomorphism. Only the prescribed
boundary points are identified.

For compact $K\subset B$, its inverse image in $Z$ is the image of
$$
 f^{-1}(K)\sqcup(D_0\cap K),
$$
which is compact. Thus $\pi$ is proper. A proper local diffeomorphism
has finite fibers and evenly covered neighborhoods, and its nonempty
image is both open and closed. It is therefore a finite covering of $B$.
Since $B$ is simply connected, each covering component maps
diffeomorphically onto $B$ \cite[Section~1.3]{Hatcher2002}.

The attached copy of $D_0$, being connected and nonempty, lies in one
covering component $C_0$. Under $C_0\cong B$, it is precisely $D_0$,
so the part of $E$ in $C_0$ is $A_0$. Every other component is a full
copy of $B$ contained in $\operatorname{int}E$. Their inverse maps
give the asserted disjoint embedded summands.
\end{proof}

\begin{corollary}[Morse domains]\label{cor:morse-completion}
Let $q_0(x,y)=-|x|^2+|y|^2$ on $\R^p\times\R^{n-p}$, and choose
$\eps>0$ and $\sqrt\eps<\rho<R$. For $\lambda\in\{1,-1\}$ put
$$
 A_\lambda=\{\lambda(q_0+\eps)\ge0\}\cap B_R,
 \qquad L=\{q_0=-\eps\}\cap B_R.
$$
Let $G:\Omega^n\to\R^n$ be a compact immersed filling whose entire
boundary over $B_R$ maps diffeomorphically onto $L$, with interior on
the $A_\lambda$-side. If $\lambda=1$, assume $p\ge2$. Then, over $B_R$,
$$
 G^{-1}(B_R)\cong A_\lambda\sqcup\bigsqcup_{j=1}^d B_R.
$$
The inverse branches restrict over $\overline B_\rho$ to disjoint
compact embedded manifolds with corners.
\end{corollary}

\begin{proof}
The restriction of $G$ to $G^{-1}(B_R)$ is proper, since $\Omega$ is
compact. The level $q_0=-\eps$ is regular. Its two closed sides are
therefore smooth domains, and the complement of
$\operatorname{int}A_\lambda$ is the side to attach in
Lemma~\ref{lem:completion}.

The domain $U=\{q_0>-\eps\}\cap B_R$ is nonempty and star-shaped
about zero. For $p\ge1$, the other open side
$V=\{q_0<-\eps\}\cap B_R$ is diffeomorphic to
$S^{p-1}\times B_b^{n-p}\times(0,1)$, where $b^2=(R^2-\eps)/2$.
Indeed, writing $x=r\theta$, its defining inequalities are
$$
 |y|^2+\eps<r^2<R^2-|y|^2,
$$
and the position of $r^2$ between these endpoints supplies the last
coordinate. Hence the side to attach is connected: it is the closure
of $V$ when $\lambda=1$, and of $U$ when $\lambda=-1$. If $p=0$ in the
latter case, $A_{-1}=L=\varnothing$ and the conclusion is a union of
full balls. Lemma~\ref{lem:completion} gives the decomposition.

All inverse branches are defined on the larger open ball $B_R$.
Their restrictions to $\overline B_\rho$ are compact and remain
disjoint, including on the cutting sphere. The genuine face
$q_0=-\eps$ meets that sphere transversely: dependence of the normals
$(-x,y)$ and $(x,y)$ would force $y=0$ and $\rho^2=\eps$.
Thus the restricted pieces are manifolds with corners.
\end{proof}

\begin{remark}\label{rem:connected-side}
Connectedness of the attached side is essential. For $p=1$, write
$a(y)=\sqrt{|y|^2+\eps}$ and consider the two domains
$$
 \{x\ge-a(y)\}\cap B_R,\qquad
 \{x\le a(y)\}\cap B_R,
$$
with inclusion on each. Their disjoint union is a proper immersion
with one copy of $L$ as boundary and interior on the $U$-side, but
neither source piece contains a full-ball inverse branch. Attaching
the two components of the $V$-side completes different covering
components.
\end{remark}

The decomposition is local over $B_R$, not a decomposition into global
components of $\Omega$. It does not require the individual sides to be
simply connected: the covering is trivialized over the whole ball.
The statement is unchanged by a target diffeomorphism $\Theta$.
In particular, it applies over a Morse chart $\Theta(B_R)$, and the
larger chart provides inverse branches near the closed cutting domain
$\Theta(\overline B_\rho)$.

\subsection{The critical extension}\label{subsec:critical-extension}

For regular values $a<b$, an \emph{immersed band} for $F$ over $[a,b]$
is a compact manifold with corners $Z^{n+1}$ and an immersion
$\mathcal G:Z\to\R^n\times[a,b]$. Its boundary consists of the end
faces over $a,b$, which fill the corresponding sections, and a lateral
face identified by a diffeomorphism $\beta$ with $h^{-1}([a,b])$.
On that face, $\mathcal G=F\circ\beta$ with the prescribed outward
coorientation.
These identifications agree on the end sections. The only corners are
the transverse intersections of the lateral face with the time end
faces. A boundary face may have several connected components.

\begin{proposition}\label{prop:critical-extension}
For some $\tau_*>0$, suppose $h$ has a single critical point in
$h^{-1}([-\tau_*,\tau_*])$, nondegenerate and of height zero, whose
image has a unique preimage. For every sufficiently small $\eps>0$,
any immersed filling of the section at $-\eps$ is the bottom face of
an immersed band over $[-\eps,\eps]$. The upper bound on $\eps$
depends only on $F$ and the fixed critical neighborhood, not on the
incoming filling.
\end{proposition}

\begin{proof}
Put $a=-\eps$, $b=\eps$, and write the incoming filling as
$G:\Omega^n\to\R^n$. The construction replaces its partial graph
domain while leaving the other local sheets intact.

\smallskip\noindent\emph{1. The incoming decomposition.}
Choose an isolated graph cylinder as in \eqref{eq:critical-graph},
and choose nested Morse-coordinate balls compactly contained in it.
Fix
$$
 K=\Theta(\overline B_\rho)
$$
and a larger open domain $\Theta(B_R)$, $0<\rho<R$, with fixed annular
neighborhoods of $\partial K$. These choices precede that of $\eps$.
Shrink the height window so that its part over $\Theta(B_R)$ meets only
the isolated graph, and require $\eps<\rho^2$.

Over $\Theta(B_R)$ the entire boundary of the incoming filling is
one copy of $f=-\eps$. The filling occupies the closed side
$\lambda(f+\eps)\ge0$. By \eqref{eq:index-restriction} and
Corollary~\ref{cor:morse-completion}, the inverse image of this larger
open domain is a disjoint union of full copies and the prescribed
partial domain. In particular, there is an embedding
\begin{equation}\label{eq:partial-inverse}
 i:P_0\longrightarrow\Omega,\qquad
 P_0=\{z\in K:\lambda(f(z)+\eps)\ge0\},\qquad G\circ i=\Id,
\end{equation}
whose inverse branch extends across the artificial cutting boundary.
On $f=a$ it agrees with the given source identification. The domain
$P_0$ may be disconnected or empty; every other piece over $K$ is a
full ball in $\operatorname{int}\Omega$.

\smallskip\noindent\emph{2. The exterior flow.}
Choose a smooth compactly supported horizontal vector field $Y$ which
vanishes near the critical point and outside the graph neighborhood,
and satisfies $df(Y)=1$ on a fixed annulus containing all the chosen
neighborhoods of $\partial K$. This can be obtained from
$\nabla f/|\nabla f|^2$ by a cutoff. Let $\psi_s$ be its ambient flow.
After decreasing a fixed permissible time, on a smaller annulus still
containing $\partial K$ and its buffers we have
\begin{equation}\label{eq:stationary-annulus}
 f(\psi_s(z))=f(z)+s.
\end{equation}
All fixed order derivatives of $\psi_s$ are uniformly bounded there.

On the regular part of a slightly larger fixed height window, with
a smaller critical neighborhood deleted, choose a height transverse
field $V$, with $dh(V)=1$, agreeing with $j_*Y$ on the corresponding
annulus. A partition of unity preserves the affine condition
$dh(V)=1$. Extend it, with a cutoff near the critical point and outside
the fixed height interval, to a smooth field $W$ on the compact source
$M$.

To control both forward and backward exterior trajectories, choose graph
domains $O_-\Subset O_0\Subset\operatorname{int}K$, containing the
critical point, and fixed numbers $0<\tau_1<\tau_2<\tau_*$. Arrange
$O_0\subset\operatorname{int}\psi_s(K)$ for the small times in use.
Put
$$
 E_* = \{|h|\le\tau_1\}\setminus j(O_0),\qquad
 N = \{|h|<\tau_2\}\setminus j(\overline O_-).
$$
Arrange $E_*\Subset N$ and $W=V$ near $\overline N$, so $dh(W)=1$
on $N$. In the metric $g_0=F^*g_{\mathrm E}$, set
$$
 d_* =\operatorname{dist}_{g_0}(E_*,M\setminus N)>0,
 \qquad B_* =\sup_M|W|_{g_0}<\infty.
$$
If $E_*$ is empty, the exterior construction is vacuous and $d_*$ may
be any positive number. Otherwise choose $\eps<\tau_1$ and
$2\eps B_*<d_*$. Every exterior point considered below lies in $E_*$,
so its trajectory under the global flow $\Phi_s$ of $W$ stays in $N$
for $|s|\le2\eps$. In particular, these trajectories are height
transverse in both time directions.

For $0\le s\le2\eps$ define the actual exterior cut section
$$
 E_s=h^{-1}(a+s)\setminus
 \{j(z):z\in\operatorname{int}\psi_s(K),\ f(z)=a+s\}.
$$
The moving seam is
$$
 S_s=\{j(\psi_s(z)):z\in\partial K,\ f(z)=a\}.
$$
On a neighborhood of that seam, agreement of the vector fields gives
\begin{equation}\label{eq:matching-flows}
 \Phi_s(j(z))=j(\psi_s(z)).
\end{equation}
The field $W+\partial_s$ is tangent to the moving seam, so uniqueness
prevents trajectories from crossing it in either time direction. Hence
\begin{equation}\label{eq:exterior-diffeomorphism}
 \Phi_s:E_0\longrightarrow E_s
 \quad\text{is a diffeomorphism.}
\end{equation}
Surjectivity follows from the backward flow on $E_s$. In the definite
Morse cases the seam is empty and the component born or lost lies
entirely inside the moving domain.

\smallskip\noindent\emph{3. The boundary correction.}
Consider the smooth auxiliary maps on the fixed compact source
$$
 \mathcal X_s(p)=\psi_{-s}(X(\Phi_s(p))),\qquad \mathcal X_0=X.
$$
Smoothness on a fixed compact set gives
\begin{equation}\label{eq:global-displacement}
 \|\mathcal X_s-X\|_{C^1(M,g_0)}\le C|s|.
\end{equation}

On the incoming exterior $E_0$ put
$D_t=(\mathcal X_{t-a}-X)|_{E_0}$. It vanishes identically on a fixed open
annulus by \eqref{eq:matching-flows}. Extend it by zero over the inner
graph part of $P=h^{-1}(a)$. This gives a smooth function on all of
$P$, with
\begin{equation}\label{eq:boundary-displacement}
 \|D_t\|_\infty+\|dD_t\|_{X_a^*g_{\mathrm E}}\le C\eps.
\end{equation}
The metric in this estimate is exactly the restriction of $g_0$:
for $v\in TP=\ker dh$, $|DF(v)|=|DX(v)|$. The zero extension is smooth
because its transition lies in the annulus where $D_t=0$.

Choose $\rho<r_+<r_{\mathrm{stat}}<R$ with the outer radii in this
zero annulus, and put $K_+=\Theta(\overline B_{r_+})$. By isolation of
the graph and the zero extension,
$$
 X(\operatorname{supp}D_t)
 \subset\R^n\setminus\Theta(B_{r_{\mathrm{stat}}}).
$$
Thus the support has a fixed positive target distance from $K_+$:
\begin{equation}\label{eq:support-gap}
 d_0=\operatorname{dist}\bigl(K_+,
           \R^n\setminus\Theta(B_{r_{\mathrm{stat}}})\bigr)>0.
\end{equation}

Use the incoming collar \eqref{eq:critical-collar} and set
\begin{equation}\label{eq:critical-correction}
 \widehat G_t=G+\chi(\sigma/\delta)D_t\circ\pi
\end{equation}
there, keeping $G$ elsewhere. In $g=G^*g_{\mathrm E}$,
\begin{equation}\label{eq:critical-smallness}
 |D\widehat G_t-DG|_g
 \le C(\eps+\eps/\delta)
 \le C'(\eps+\sqrt\eps)<\frac12
\end{equation}
for sufficiently small $\eps$. Hence every $\widehat G_t$ is an
immersion. Moreover,
\begin{equation}\label{eq:protected-preimage}
 \widehat G_t=G\quad\text{on }G^{-1}(K_+).
\end{equation}
Indeed, a collar point $y$ with nonzero correction has a boundary footpoint
$p=\pi(y)$ in the support in \eqref{eq:support-gap}, and
$|G(y)-X(p)|=\sigma(y)\le\delta$. Taking
$\delta<d_0$ excludes its image from $K_+$. Thus \eqref{eq:protected-preimage} holds on every sheet over $K_+$,
and the maps agree on a neighborhood of $G^{-1}(K)$.

\smallskip\noindent\emph{4. Replacement of the partial domain.}
The interior and one-sided boundary inverse charts show that
$i(P_0\cap\operatorname{int}K)$ is relatively open in $\Omega$. Thus
$$
 Q=\Omega\setminus i(P_0\cap\operatorname{int}K)
$$
is compact. It is a manifold with corners, with genuine boundary
$E_0$ and artificial face $i(J)$, where
$$
 J=\{z\in\partial K:\lambda(f(z)+\eps)\ge0\}.
$$
These are embedded source faces by Step~1. No full interior copy is
removed.

Replace the removed piece by
\begin{equation}\label{eq:replacement-domain}
 \mathcal E=\{(z,t)\in K\times[a,b]:
           \lambda(f(\psi_{t-a}(z))-t)\ge0\},
\end{equation}
with map $T(z,t)=(\psi_{t-a}(z),t)$. On $Q\times[a,b]$ use
\begin{equation}\label{eq:exterior-map}
 (y,t)\longmapsto(\psi_{t-a}(\widehat G_t(y)),t).
\end{equation}
Both differentials are block triangular with invertible spatial blocks
and last diagonal entry one.

On the seam annulus, \eqref{eq:stationary-annulus} yields
\begin{equation}\label{eq:stationary-domain}
 f(\psi_{t-a}(z))-t=f(z)+\eps.
\end{equation}
Consequently the artificial face of $\mathcal E$ is $J\times[a,b]$.
Glue it to $i(J)\times[a,b]$ by $(z,t)\sim(i(z),t)$. Equations
\eqref{eq:partial-inverse} and \eqref{eq:protected-preimage} show that
the maps and their domain descriptions agree in open inverse
coordinate neighborhoods of the seam, not just pointwise on it.

The genuine boundary is smooth: near the critical point $Y=0$, so
its defining function has time derivative $-\lambda\ne0$; away from
that point its spatial differential is nonzero. It meets the cutting
face transversely by Corollary~\ref{cor:morse-completion}, and the time
end faces transversely because $a,b$ are regular. At a triple
intersection, \eqref{eq:stationary-domain} makes the spatial defining
functions independent of time, so their differentials and $dt$ are
independent. Gluing removes the artificial face and its corners.

The pieces and glued faces are compact, and the closed face
identification gives a compact Hausdorff quotient. The inverse charts
at the seam make it a smooth manifold with corners and define an
immersion, denoted by $(Z,\mathcal G)$.

\smallskip\noindent\emph{5. Boundary identification and coorientation.}
On the exterior lateral boundary define
$$
 \beta(p,t)=\Phi_{t-a}(p),\qquad p\in E_0,
$$
and on the genuine boundary of $\mathcal E$ define
$$
 \beta(z,t)=j(\psi_{t-a}(z)).
$$
They agree on an open seam neighborhood by
\eqref{eq:matching-flows}. The first map covers exactly the exterior
part of each source section by \eqref{eq:exterior-diffeomorphism};
the second covers exactly its part in the isolated graph. After gluing,
these maps therefore define a bijection onto $h^{-1}([a,b])$. Its inverse on
the graph, including at the critical point, is explicitly
$$
 j(z')\longmapsto
 \bigl(\psi_{f(z')-a}^{-1}(z'),f(z')\bigr).
$$
Outside the graph it is the backward source flow. These formulas prove
that $\beta$ is a diffeomorphism. Equations
\eqref{eq:critical-correction} and \eqref{eq:exterior-map} give
$\mathcal G=F\circ\beta$ on the exterior, and the same identity is
immediate on the graph.

On the graph the inequality \eqref{eq:graph-side} gives exactly the
prescribed outward normal, including at the critical point. On the
exterior the spatial outward normal starts as $\mu_a$ and continues
through immersions, so it is the prescribed $\mu_t$. For a regular
boundary parameterization $(X_t(p),t)$, the full outward normal is a
positive multiple of
$$
 \bigl(\mu_t,-\langle\partial_tX_t,\mu_t\rangle\bigr).
$$
Since this vector is $\nu/\alpha$ along the original immersion, its
coorientation is also the required one.

\smallskip\noindent\emph{6. End faces and choice of the interval.}
At $t=a$ we have $D_a=0$, $\psi_0=\Id$, and $\mathcal E_a=P_0$.
Thus the bottom face restores exactly the original source, map, and
boundary identification. At $t=b$ all artificial faces have been
removed and the genuine section is regular, so the top face is a
compact smooth immersed filling with the asserted boundary and
coorientation.If the incoming filling is empty, the collar and exterior
correction are absent and the same local construction applies.

After fixing the neighborhoods, fields, and permissible flow times,
choose $\eps$ so that
$$
 \eps<\min\{\rho^2,\tau_1\},\qquad
 2\eps B_*<d_*,\qquad
 c_0\sqrt\eps<d_0,\qquad
 C'(\eps+\sqrt\eps)<\tfrac12,
$$
and so that $2\eps$ is smaller than those permissible flow times.
All constants depend only on the prescribed immersion and the fixed
neighborhoods. The interval is therefore independent of the incoming
filling.
\end{proof}

\section{Proof of the filling theorem}\label{sec:global}

We first glue the bands from Section~\ref{sec:critical} along regular
end sections. The source boundary identification determines which
branches meet at a self-intersection.

\begin{lemma}[Gluing bands]\label{lem:end-gluing}
Let $Z_-$ and $Z_+$ be immersed bands for the same cooriented immersion
$F$ over $[a_-,a]$ and $[a,a_+]$, where $a$ is a regular height.
A diffeomorphism between their end faces at $a$ that preserves the
immersions and the source boundary identifications glues them to a
smooth immersed band over $[a_-,a_+]$.
\end{lemma}

\begin{proof}
The quotient is compact and Hausdorff because the end faces are closed.
At an interior point of the seam, the immersion charts on the two
pieces give opposite time half-neighborhoods in the target and glue
to a full neighborhood.

At a boundary point of the seam, the source identifications select
the same local branch of $F$. Regularity of the height gives target
coordinates $(u,r,t)$ in which that branch is $r=0$ and its prescribed
interior is $r\ge0$. The corner charts are
$$
 \{r\ge0,\ t\le a\},\qquad \{r\ge0,\ t\ge a\},
$$
and their union is the half-neighborhood $r\ge0$. These target charts
define a smooth immersion on the quotient and a smooth identification
of its lateral boundary with the source of $F$. The common end face
and its corners disappear.
\end{proof}

\begin{proof}[Proof of Theorem~\ref{thm:filling}]
If $M$ is empty, take $W$ empty. Henceforth assume that $M$ is nonempty.
First suppose $F$ has the generic height described in
Proposition~\ref{prop:genericity}. Let its finitely many critical
values be $c_1<\cdots<c_N$. Before constructing any filling, choose
isolated graph neighborhoods and sufficiently small, pairwise
disjoint critical intervals
$$
 [c_i-\eps_i,c_i+\eps_i],\qquad 1\le i\le N,
$$
for Proposition~\ref{prop:critical-extension}. The choices are
possible in advance because its constants do not depend on the
incoming filling. Take regular heights below the minimum and above
the maximum of $h$.

Start with the empty filling below the minimum. On each intervening
compact regular interval, identify the sections by a height-transverse
source flow. They have $H_t>0$ by \eqref{eq:slice-form}. The constants
in Lemma~\ref{lem:local-extension} are uniform on this fixed
prescribed family. Subdivide into finitely many sufficiently short
intervals and construct the corresponding families of fillings.
Each short family gives a compact band with map
$$
 (y,t)\longmapsto(G_t(y),t),
$$
whose differential is invertible by its block triangular form.
Its lateral boundary is the exact source family of height sections.
At a critical interval, apply
Proposition~\ref{prop:critical-extension} to the preceding endpoint
filling and retain the stated source boundary identification for
the filling at the upper endpoint.

The filling obtained at the end of each step, together with its
source boundary identification, serves as the initial filling for
the next step. Lemma~\ref{lem:end-gluing} therefore
glues the resulting bands smoothly. The estimates are unaffected by
the boundary identifications. Indeed,
if the identification at height $a$ is
$\beta:\partial\Omega\to h^{-1}(a)$, the identity
$G|_{\partial\Omega}=X_a\circ\beta$ makes $\beta$ an isometry for
the two induced boundary metrics. Pulling back the prescribed
boundary corrections thus preserves their intrinsic differential
norms exactly.

There are finitely many critical intervals. Every remaining regular
interval has a positive uniform step size depending only on its
prescribed boundary family; reapplication of the collar lemma uses
the same bound because the boundary at each step is still that
family. Hence the construction uses finitely many compact pieces.
Its quotient is compact, with a smooth immersion into
$\R^{n+1}$ and with lateral boundary identified with all of $M$.
Indeed, the source height portions meet only in their common end
sections, where the source points are identified exactly. The resulting
lateral map is a diffeomorphism to $M$ by the charts
in Lemma~\ref{lem:end-gluing}.
Every intermediate time face and every artificial spatial face has
been removed by gluing.

Above the last critical value the boundary section is empty. The corresponding filling must also be empty: otherwise it would be a nonempty
compact manifold without boundary admitting a local diffeomorphism
to $\R^n$, whose image would be both open and compact. The bottom
face was empty from the start. Thus the assembled manifold $W$ has
smooth boundary exactly $M$, its map $G$ restricts to $F$ with the
specified parameterization, and its outward normal is $\nu$.
This proves the theorem for the generic immersion.

For the original immersion, Proposition~\ref{prop:genericity} gives
an arbitrarily small perturbation $\widetilde F$ and a straight-line
regular homotopy back to $F$ preserving $P_1>0$. Since
$\operatorname{tr}P_1=(n-1)H$, this homotopy has $H>0$ with its
transported normal. The generic case provides a
filling of $\widetilde F$ inducing its prescribed outward normal, and
Corollary~\ref{cor:homotopy} transfers it to the original $F$, with
the exact original boundary map.
\end{proof}

\section{Constant scalar curvature}
\label{sec:rigidity}

Write $H_1=H/n$ and $c=\Scal/[n(n-1)]$. All hypersurface integrals
are taken over the source with its induced measure.

The positivity of $\Scal$ and $P_1$, the choice of a global normal with
$H>0$, and the resulting strict $(n-1)$-convexity are also recorded in
\cite[Lemma~2.1 and Remark~2.2]{FaguiLi2026}. We include the standard
approach to fix our conventions.

\begin{lemma}\label{lem:admissibility}
Let $F:M^n\to\R^{n+1}$, $n\ge2$, be an immersion of a nonempty closed
connected manifold with constant scalar curvature. Then $\Scal>0$, and
$F$ has a global unit normal for which $H>0$, $P_1>0$, and
$H_1\ge\sqrt c$.
\end{lemma}

\begin{proof}
Fix $a\in\R^{n+1}$, and let $p$ maximize
$f=\frac12|F-a|^2$. Put $\rho=|F(p)-a|>0$ and choose a local unit normal
with $\nu(p)=(F(p)-a)/\rho$. With this choice,
$$
 \nabla^2 f(X,Y)
 =\langle X,Y\rangle-\langle F-a,\nu\rangle\langle AX,Y\rangle.
$$
At $p$ the Hessian is nonpositive, so $A(p)\ge\rho^{-1}\Id$.
The Gauss equation therefore gives
$\Scal(p)\ge n(n-1)\rho^{-2}>0$. Constancy implies $\Scal>0$ on $M$.

For either local choice of normal, $\Delta_MF=-H\nu$ and
$|\Delta_MF|^2=H^2=\Scal+|A|^2>0$. Hence
$$
 \nu=-\frac{\Delta_MF}{|\Delta_MF|}
$$
defines a smooth global unit normal and gives $H=|\Delta_MF|>0$.
Since $H>|A|\ge|\kappa_i|$ for every principal curvature $\kappa_i$,
all eigenvalues $H-\kappa_i$ of $P_1$ are positive. Finally,
$$
 H_1^2-c=\frac{|A-H_1\Id|^2}{n(n-1)}\ge0.
$$
\end{proof}

We recall Ros's rigidity argument \cite[Theorems~1 and~2]{Ros1987},
including the extension to immersed fillings attributed there to
Pinkall. The proof uses the Minkowski identity \cite{Hsiung1954,Ros1987}
and the boundary Heintze--Karcher inequality \cite{HeintzeKarcher1978}.
We derive the latter from Reilly's identity \cite{Reilly1977}, following
Ros's Dirichlet method.

\begin{proposition}[Ros--Pinkall]\label{prop:ros-pinkall}
Let $F:M^n\to\R^{n+1}$, $n\ge2$, be an immersion of a nonempty closed
connected manifold with constant scalar curvature. If $F$ extends to an
immersion of a compact manifold bounded by $M$, then $F$ is a
diffeomorphism onto a round sphere.
\end{proposition}

\begin{proof}
Let $G:W^{n+1}\to\R^{n+1}$ be the extension, with $\partial W=M$.
Equip $W$ with the flat metric $g_W=G^*g_{\mathrm E}$. Its boundary
metric is the metric induced by $F$. No component of $W$ can be closed:
the image of such a component under a local diffeomorphism would be both
open and compact in $\R^{n+1}$. Since $M$ is connected, $W$ is connected.
Write $V=\operatorname{Vol}(W)>0$, and let $\eta$ be its outward unit normal.

The normal $dG(\eta)$ agrees with the normal in
Lemma~\ref{lem:admissibility}. To see this, maximize
$\frac12|G-a|^2$ over $W$. A maximum cannot be interior because $G$ is a
local diffeomorphism. At a boundary maximum the radial normal points
outward: differentiation in the inward normal direction has nonpositive
sign, and the nonzero radial vector is orthogonal to the boundary.
The boundary Hessian approach in Lemma~\ref{lem:admissibility} therefore
gives a positive definite shape operator for the outward normal there.
Because $\Scal>0$ makes the outward mean curvature nowhere zero,
connectedness gives $H>0$ everywhere for this normal. We henceforth
identify $dG(\eta)$ with $\nu$.

Set $u=\langle F-a,\nu\rangle$. The pullback of the ambient position
field $G-a$ has covariant derivative $\Id$ on $W$, so the divergence
theorem gives
\begin{equation}\label{eq:rigidity-volume}
 \int_M u\,d\mu=(n+1)V.
\end{equation}
For the tangent field $Z=(F-a)^\top$ on $M$, direct differentiation gives
$\nabla Z=\Id-uA$. The Codazzi equation implies
$\operatorname{div}P_1=0$: in local components,
$$
 \nabla_i(P_1)^i{}_j=\nabla_jH-\nabla_iA^i{}_j=0.
$$
Consequently
$$
 \operatorname{div}(P_1Z)
 =\operatorname{tr}P_1-u\operatorname{tr}(P_1A)
 =(n-1)H-u\Scal.
$$
After integration, constancy of $\Scal=n(n-1)c$ yields
\begin{equation}\label{eq:rigidity-minkowski}
 \int_M H_1\,d\mu=c\int_Mu\,d\mu=c(n+1)V.
\end{equation}

As in the proof of \cite[Theorem~1, p.~449]{Ros1987}, smooth Dirichlet
theory supplies a solution $v\in C^\infty(W)$ of
$$
 \Delta v=1\quad\text{on }W,\qquad v=0\quad\text{on }M,
$$
where $\Delta=\operatorname{tr}\nabla^2$, and set
$q=\partial_\eta v$. Then $\int_Mq\,d\mu=V$. For
$$
 Y=(\Delta v)\nabla v-
       \bigl(\nabla^2v(\nabla v,\cdot)\bigr)^\sharp,
$$
the contracted Ricci identity gives
$$
 \operatorname{div}Y
 = (\Delta v)^2-|\nabla^2v|^2
   -\operatorname{Ric}_{W}(\nabla v,\nabla v).
$$
On $M$ we have $\nabla v=q\eta$ and
$\Delta v=\nabla^2v(\eta,\eta)+Hq$, hence
$\langle Y,\eta\rangle=Hq^2$. Flatness and the pointwise inequality
$|\nabla^2v|^2\ge(\Delta v)^2/(n+1)$ imply
$$
 \int_MHq^2\,d\mu
 =\int_W\bigl(1-|\nabla^2v|^2\bigr)\,dV_W
 \le\frac n{n+1}V.
$$
Since $H>0$, the weighted Cauchy--Schwarz inequality gives
$$
 V^2\le
 \left(\int_MHq^2\,d\mu\right)
 \left(\int_M\frac1H\,d\mu\right)
 \le\frac n{n+1}V\int_M\frac1H\,d\mu.
$$
Thus
\begin{equation}\label{eq:rigidity-hk}
 \int_M\frac1{H_1}\,d\mu\ge(n+1)V.
\end{equation}

Lemma~\ref{lem:admissibility} gives $H_1^2\ge c$, whereas
\eqref{eq:rigidity-minkowski} gives the reverse integral inequality:
$$
 \int_M\frac1{H_1}\,d\mu
 \le\frac1c\int_M H_1\,d\mu=(n+1)V.
$$
Equality in conjunction with \eqref{eq:rigidity-hk} forces the continuous
nonnegative function $H_1/c-1/H_1$ to vanish. Therefore $H_1=\sqrt c$ and
$A=\sqrt c\,\Id$ everywhere.

Put $r=c^{-1/2}$. Since $D(F-r\nu)=0$ and $M$ is connected,
$F-r\nu=a_0$ for a fixed $a_0\in\R^{n+1}$. The map $F$ thus takes values
in the round sphere $S^n_r(a_0)$ and is a local diffeomorphism onto it.
Its image is open and closed, so it is surjective. Compactness makes
$F:M\to S^n_r(a_0)$ a proper local diffeomorphism and hence a covering.
Since $n\ge2$ and $M$ is connected, the simple connectedness of $S^n$
makes this covering one-sheeted. The resulting diffeomorphism is an
isometry for the induced metric.
\end{proof}

\begin{proof}[Proof of Theorem~\ref{thm:csc}]
For a compact connected hypersurface immersion with constant scalar
curvature, Lemma~\ref{lem:admissibility} supplies a global normal for which
$P_1>0$. Theorem~\ref{thm:filling} then supplies a compact smooth immersed
filling, and Proposition~\ref{prop:ros-pinkall} gives the round sphere
classification.
\end{proof}

\subsection*{Declaration of competing interest}
	The authors declare that they have no known competing financial interests in this paper.

	\subsection*{Data availability}
	No data were used for the research described in the article.
	
\subsection*{On the use of AI}
ChatGPT 6 was employed to discuss potential ideas. Kimi K3 was used for grammatical refinement, spelling correction, and suggestions on the phrasing of section and lemma headings. Neither ChatGPT 6 nor Kimi K3 was used to produce any manuscript text, arguments, proofs, or computations. The authors take full responsibility for this work.

\bibliographystyle{amsplain}

\providecommand{\bysame}{\leavevmode\hbox to3em{\hrulefill}\thinspace}
\providecommand{\MR}{\relax\ifhmode\unskip\space\fi MR }

\providecommand{\MRhref}[2]{%
  \href{http://www.ams.org/mathscinet-getitem?mr=#1}{#2}
}
\providecommand{\href}[2]{#2}

\end{document}